\ifx\ArticleWithinIssue\undefined
\documentclass{resmath_arxiv} 

\fi

\authortitleforheadfootcontent{M. V.~Pratsiovytyi, D. M.~Karvatskyi, O. P.~Makarchuk}{Infinite Bernoulli convolutions and their properties} 

\udc{514.8, 517.5, 519.21}

\title{Infinite Bernoulli convolutions generated by multigeometric series and their properties}

\thanks{The first and third authors were supported by a grant from the Simons Foundation (grant 1290607, P.M.V. and M.O.P.). The second author worked at the University of St. Andrews within the framework of the Isaac Newton Institute Solidarity Programme and also received additional support from the London Mathematical Society.}

\author{M. V.~Pratsiovytyi${}^{*,**}$, D. M.~Karvatskyi${}^{**}$, O. P.~Makarchuk${}^{***}$}
       {${}^{*}$~Dragomanov Ukrainian State University,
        Kyiv 01601. \\
                ${}^{**}$~Institute of Mathematics of NAS of Ukraine,
        Kyiv 01024. \\ 
        ${}^{***}$~Central Ukrainian State University,
        Kropyvnytskyi 25006. \\ {\it E-mails: prats4444@gmail.com, karvatsky@imath.kiev.ua, makolpet@gmail.com}}
\date{31.01.2020}

\ifx\ArticleWithinIssue\undefined
\begin{document}\fi 

\sloppy

\begin{article}

\abstracten{The paper is devoted to infinite Bernoulli convolutions generated by positive multigeometric series and to probability distributions of random variables whose digits in an even integer base-$s$ expansion with two redundant digits form a sequence of independent and identically distributed random variables.

The main objects of the article are random variables:

$\xi=\sum\limits_{n=1}^{\infty}\frac{\xi_n}{s^n}$,
where $(\xi_n)$ is a sequence of independent and identically distributed random variables taking values $0, 1, 2, \dots, s-1, s, s+1$ with probabilities
$p_0$, $p_1$, $p_2, \dots, p_{s-1}, p_s, p_{s+1}$ respectively $(3<s \in \mathbb{N})$;
$$\eta=\sum\limits_{n=1}^{\infty}\left[\frac{3\eta_{(n-1)(m+1)+1}}{s^n}+\sum\limits_{j=1}^{m}
\frac{2\eta_{(n-1)(m+1)+1+j}}{s^n}\right],$$
where $(\eta_n)$ is a sequence of independent and identically distributed random variables that take values $0$ and $1$ with probabilities $q_0>0$ and ~ ~ ~ ~ ~ ~ ~ ~  $q_1=1-q_0>0$. We study conditions under which the above random variables have absolutely continuous or singular distributions as well as topological, metric, and fractal properties of their supports. The main focus is on the case where the spectrum is a Cantorval.

For the case $s=4$, we establish necessary and sufficient conditions for singularity and absolute continuity of the distributions of $\xi$ and $\eta$, in particular when they are supported on the Guthrie–Nymann Cantorval.  For an arbitrary even $s>4$, we determine necessary and sufficient conditions for a random variable $\xi$ to admit a decomposition into the sum of two independent random variables, one of which is uniformly distributed on the unit interval and therefore absolutely continuous. Using the method of characteristic functions, sufficient conditions for singularity and necessary conditions for absolute continuity are obtained. For Cantorvals arising as spectra of the corresponding distributions, the structure and fractal properties of the boundary are investigated.} 
\smallskip
\keywordsen{infinite Bernoulli convolutions, singular distributions of random variables, the set of subsums, Cantorval, multigeometric series} 

\bigskip

\abstractua{Стаття присвячена нескінченним згорткам Бернуллі, керованим додатними мультигеометричними рядами, і розподілам випадкових величин, цифри зображення яких у системі з парною натуральною основою $s$ і двома надлишковими цифрами є незалежними та однаково розподіленими випадковими величинами.

Основними об'єктами дослідження є випадкові величини:

$\xi=\sum\limits_{n=1}^{\infty}\frac{\xi_n}{s^n}$,
де $(\xi_n)$ -- послідовність незалежних  випадкових величин, які набувають значень $0,1,2,...,s-1,s, s+1$ з ймовірностями
$p_0$, $p_1$, $p_2,...,p_{s-1}, p_s, p_{s+1}$ відповідно $(3<s \in \mathbb{N})$;
$$\eta=\sum\limits_{n=1}^{\infty}\left[\frac{3\eta_{(n-1)(m+1)+1}}{s^n}+\sum\limits_{j=1}^{m}
\frac{2\eta_{(n-1)(m+1)+1+j}}{s^n}\right],$$
де $(\eta_n)$-- послідовність незалежних однаково розподілених випадкових величин, які набувають значень $0$ та $1$ з ймовірностями $q_0>0$ і $q_1=1-q_0>0$ відповідно. Вивчаються умови абсолютної неперервності та сингулярності розподілів цих випадкових величин, тополого-метричні і фрактальні властивості їх носіїв. Основна увага зосереджена на випадках, коли спектром розподілу є канторвал.

У випадку $s=4$ для випадкових величин $\xi$ та $\eta$ знайдено необхідні і достатні умови сингулярності та абсолютної неперервності розподілів, зокрема коли вони зосереджені на канторвалі Гатрі-Німана. Для довільного парного $s>4$  знайдено необхідні і достатні умови, за яких $\xi$ розкладається в суму двох незалежних випадкових величин, одна з яких має рівномірний розподіл на одиничному відрізку, що рівносильно її абсолютній неперервності. За допомогою методу характеристичних функцій отримано достатні умови сингулярності та необхідні умови абсолютної неперервності. Для канторвалів, що є спектрами розподілів, вивчено структуру і фрактальні властивості межі.
} 
\smallskip
\keywordsua{нескінченні згортки Бернуллі, сингулярні розподіли випадкових величин, множини неповних сум, канторвал, мультигеометричні ряди} 

\bigskip

\mathsubjclass{Pri 60E05, Sec 40A05, 28A80}

\bigskip

\section{Introduction}

Singular continuous distributions of random variables are concentrated on sets of Lebesgue measure zero, often with fractal structure. Such distributions naturally arise in the study of random variables with independent digits in various systems of representation (coding) of real numbers~\cite{PF},~\cite{PM}. A typical phenomenon is that an infinite Bernoulli convolution generated by a convergent positive series has a singular distribution. We are interested in infinite Bernoulli convolutions generated by positive multigeometric series whose sets of subsums are Cantorvals. Such series and their sets of subsums are currently the subject of active research. The general topological and metric theory of subsums of absolutely convergent series is not yet sufficiently developed and is still at a constructive stage of study \cite{GP25}.

This paper is devoted to probability measures on sets of subsums of multigeometric series, in particular those that are Cantorvals.

We recall some key definitions and facts that will be used later on, and also formulate the problems that are partially addressed in this paper.
Let $A_s=\{0,1,\dots,s-1\}$ be an alphabet, and let $L_s=A_s\times A_s\times \dots \times A_s \times \dots$ be the space of sequences with elements from this alphabet for some $2 \leq s \in \mathbb{N}$.

A \textit{subsum} of a convergent series $r_0=v_1+v_2+\dots+v_n+\dots$, determined by a set $M\subset \mathbb{N}=\{1,2,\dots\}$, is the number
$$
x(M)=\sum\limits_{n\in M\subset \mathbb{N}} v_n=\sum\limits_{n=1}^{\infty}\varepsilon_n v_n,
\quad \text{where } \varepsilon_n=1 \text{ if } n\in M,\text{ and } \varepsilon_n=0 \text{ if } n\notin M.
$$
It is clear that each partial sum $S_n=v_1+v_2+\dots+v_n$ and each remainder of the series $r_n=v_{n+1}+v_{n+2}+\dots$ is a subsum of the series, although these do not exhaust all subsums.
The \textit{set of subsums} of a given convergent series $\sum v_n$ is the set of all its subsums, i.e.,
$$
E(v_n)=\left\{x: x=\sum_{n=1}^{\infty}\varepsilon_n v_n,\ (\varepsilon_n)\in L_2\right\}.
$$

It is well known that the set of subsums of a convergent positive series is a perfect continuum set that is symmetric with respect to its midpoint $r_0 / 2$. It may be a finite union of intervals, or it may be a nowhere dense set. Kakeya's result~\cite{kak} provides necessary and sufficient conditions for the set of subsums to belong to the first type, and sufficient conditions for it to belong to the second type. A complete topological classification of sets of subsums of convergent positive series is presented in \cite{Nymann}.

\begin{theorem}
\label{GN-Theorem}
The set of subsums $E(v_n)$ of a convergent positive series belongs to one of the following three types:

\begin{enumerate}

\item an interval or a finite union of intervals;

\item a set homeomorphic to the classical Cantor set
$$
C=\left\{x: x=\sum\limits_{k=1}^{\infty}\frac{\alpha_k}{3^k},\ \alpha_k\in\{0,2\}\ \forall k\in \mathbb{N}\right\};
$$

\item an $M$-Cantorval, that is, a set homeomorphic to
$$
T^* \equiv C \cup \bigcup_{n=1}^{\infty} G_{2n-1} \equiv [0,1]\setminus \bigcup_{n=1}^{\infty} G_{2n},
$$
where
$$
G_n=\bigcup\limits_{\alpha_1\in\{0,1\}}\dots\bigcup\limits_{\alpha_{n-1}\in\{0,1\}}
\left\{x:\frac{1}{3^n}+\sum\limits_{k=1}^{n-1}\frac{2\alpha_k}{3^k}<x<\frac{2}{3^n}+
\sum\limits_{k=1}^{n-1}\frac{2\alpha_k}{3^k}\right\}.
$$

\end{enumerate}
\end{theorem}

At present, necessary and sufficient conditions for the set of subsums to be a Cantorval or a nowhere dense set remain unknown. Criteria for these sets to be of Lebesgue measure zero in the case of nowhere density are not yet fully understood. These problems are being actively investigated for various classes of numerical series with a relatively simple structure. The most significant results in this direction were obtained for multigeometric series  
$$k_1+k_2+\dots+k_m+k_1q+\dots+k_mq+\dots + k_1q^i+\dots+k_mq^i+\dots,$$
where $k_1, k_2, \dots, k_m$ are fixed positive scalars, and $0<q<1$. In particular, conditions for the set of subsums of multigeometric series to be a Cantorval were established in \cite{Banakh}, \cite{BP17}, \cite{Bartoszewicz}, \cite{Ferdinands}, \cite{Nit15}, with further generalizations provided in \cite{KMV}, \cite{PKC}.
Some conditions for the set of subsums of a non-multigeometric series were found in \cite{VMPS19}. 

One of the simplest and most famous Cantorvals was introduced by J.~Guthrie and J.~Nymann in \cite{Guthrie 1988}, where the authors considered the set
$$
X=\left\{\sum_{n=1}^{\infty} \frac{\alpha_n}{4^n} : (\alpha_n) \in \{0, 2, 3, 5 \}^{\mathbb{N}} \right\},
$$
which is the set of subsums of the series
\begin{equation*}
\label{GNS}
\frac{3}{4}+\frac{2}{4}+\frac{3}{4^2}+\frac{2}{4^2}+\frac{3}{4^3}+\frac{2}{4^3}+\dots+\frac{3}{4^k}+\frac{2}{4^k}+\dots.
\end{equation*}
Its properties were also studied in \cite{Bielas}, \cite{Glab}. In particular, the latter paper investigates the set $U$ of numbers having a unique representation as a subsum of the Guthrie--Nymann series. It was shown that $U$ is dense in $X$, but has empty interior. Some modifications of the set $X$ were also studied in \cite{PKS}.

By an \emph{infinite Bernoulli convolution generated by a convergent series} ~ ~ ~ $v_1+v_2+\dots+v_n+\dots$, we mean the distribution of the random variable
$
\eta=v_1\eta_1+v_2\eta_2+\dots+v_n\eta_n+\dots,
$
where $(\eta_n)$ is a sequence of independent and identically distributed random variables taking the values $0$ and $1$ with probabilities $q_0>0$ and $q_1=1-q_0>0$, respectively.

According to the Jessen--Wintner theorem~\cite{win}, the distribution of $\eta$ is either purely absolutely continuous (probability that $\eta$ takes a value in a zero Lebesgue measure set is $0$) or purely singular (there exists a set of Lebesgue measure zero such that $\eta$ belongs to it with probability $1$). In general, necessary and sufficient conditions for an infinite Bernoulli convolution to be singular or absolutely continuous are still unknown. This problem is known as the problem of refining or deepening the Jessen--Wintner theorem. In some special classes of random variables, in particular for infinite Bernoulli convolutions, it can sometimes be solved completely \cite{MK24}, \cite{MakarchukKhaletsky2025}.

\section{Representations of numbers in a system with a redundant alphabet}

Throughout the paper, we use the following notation: $m$ is a fixed positive integer parameter, $s=2m+2$, $A=\{0,1,\dots,s,s+1\}$ is the redundant alphabet, $L=A\times A\times \dots$ is the set of sequences of elements of the alphabet. In a similar manner, we define $A_0\equiv A\setminus\{1,s\}$ and $L_0=A_0\times A_0\times\cdots$.

If
$
x=\sum\limits_{n=1}^{\infty}s^{-n}\alpha_n\equiv \Delta_{\alpha_1\alpha_2\dots\alpha_n\dots},
$
then the latter symbolic expression will be called the \emph{$\Delta$-representation} of $x$. Parentheses in the $\Delta$-representation of a number will indicate a period. A \emph{cylinder} of rank $m$ with base $c_1c_2\dots c_m$ is the set $\Delta_{c_1c_2\dots c_m}$ of all numbers $x$ having a $\Delta$-representation
$
x=\Delta_{\alpha_1\alpha_2\dots\alpha_n\dots},
$
such that $\alpha_k=c_k$ for $k=\overline{1,m}$. The cylinder $\Delta_{c_1c_2\dots c_m}$ is the interval $[a_m;b_m]$, where
$
a_m=\sum\limits_{k=1}^{m}\frac{c_k}{s^k},
b_m=a_m+\frac{s+1}{s^m(s-1)},
$
and
$$
\Delta_{c_1c_2\dots c_m}
=
\Delta_{c_1c_2\dots c_m0}\cup
\Delta_{c_1c_2\dots c_m1}\cup\dots\cup
\Delta_{c_1c_2\dots c_m[s+1]}.
$$

%
%
%

Note that numbers from the interval $\left[0,\frac{s+1}{s-1}\right]$ in the numeration system with base $s$ and two redundant digits $s$ and $s+1$ may have a unique representation, a finite, countable, or continuum many distinct representations. Moreover, any pair of consecutive digits $(a,s+j)$ in a representation of a number can be replaced, without changing the value of the corresponding series, by the alternative pair $(a+1,j)$, where $a\in \{0,1,\dots,s\}$ and $j\in \{0,1\}$. Therefore, the number $x=\Delta_{(0s)}$ has a continuum number of distinct representations, since the pairs $(0,s)$ and $(1,0)$ are interchangeable. The number $x=\Delta_{(1)}$ has countably many representations 
$
x=\Delta_{(1)}
=\Delta_{0(s)}=\Delta_{10(s)}=\Delta_{11\dots10(s)}.
$

Assuming that $s\geq 4$, we prove that the number $x_0=\Delta_{(23)}$ has a unique $\Delta$-representation.
Indeed, suppose that $x_0$ has another representation $\Delta_{\alpha_1\alpha_2\dots\alpha_n\dots}$ different from the one indicated above. Then there exists a smallest $k$ such that the digit $\alpha_k$ does not coincide with the corresponding digit in the given representation. If $k=2n-1$, then by assumption $\alpha_k\neq 2$. However,
\[
\Delta_{23\dots232(2)}=\Delta_{23\dots231(s+1)}<x_0<\Delta_{23\dots233(0)}.
\]
Thus, $\alpha_{2n-1}=2$. If $k=2n$, then by assumption $\alpha_k\neq 3$. However,
\[
\Delta_{23\dots2323(2)}=\Delta_{23\dots2322(s+1)}<x_0<\Delta_{23\dots2324(0)}.
\]
Therefore, $\alpha_{2n}=3$. Hence, the number $x_0$ has a unique $\Delta$-representation.

Almost all numbers in the interval $\left[0,\frac{s+1}{s-1}\right]$, in the sense of Lebesgue measure, have continuum many distinct $\Delta$-representations (a complete answer to the question regarding distinct representations can be found in~\cite{prv}).

\section{A random variable with independent and identically distributed digits}

Let us consider a random variable
\[
\xi=\sum\limits_{n=1}^{\infty}s^{-n}\xi_n\equiv\Delta_{\xi_1\xi_2\dots\xi_n\dots},
\]
where $(\xi_n)$ is a sequence of independent random variables taking the values
$0,1,2,\dots,s-1,s,s+1$ with probabilities
$p_0,p_1,p_2,\dots,p_{s-1},p_s,p_{s+1}$, respectively, where $p_i<1$.
It is clear that the set $E_\xi$ of values of the random variable $\xi$ coincides with the interval $\left[0,\frac{s+1}{s-1}\right]$.
We are interested in revealing the type of the distribution of $\xi$ and establishing structural, topological, metric, and fractal properties of its spectrum (the set of growth points of the distribution function).

According to the Jessen--Wintner theorem, the random variable $\xi$ has either a purely singular distribution or a purely absolutely continuous distribution. In the general setting, the problem of finding necessary and sufficient conditions under which the distribution is singular is difficult. This paper is devoted to a partial solution of this problem using the method of characteristic functions and the method of distinguishing the absolutely continuous component. This approach was effectively applied in the study of similar objects in~\cite{P1996}.

Note that in the case $m=1$ and $p_1=0=p_4$, the spectrum of the distribution of $\xi$ coincides with the Guthrie--Nymann Cantorval. For arbitrary $m \in \mathbb{N}$, under the condition $p_1=0=p_s$, the spectrum of the distribution of $\xi$ coincides with the set of subsums of the series
\[
\frac{3}{s}+\underbrace{\frac{2}{s}+\dots+\frac{2}{s}}_m+\frac{3}{s^2}+\underbrace{\frac{2}{s^2}+\dots+\frac{2}{s^2}}_m+\dots+
\frac{3}{s^n}+\underbrace{\frac{2}{s^n}+\dots+\frac{2}{s^n}}_m+\dots,
\]
which is also a Cantorval and was partially studied in \cite{Banakiewicz}. The infinite Bernoulli convolution generated by this series is closely related to the distribution of the random variable $\xi$. We describe this relationship in detail below.

%
%
%

\section{Absolutely continuous distributions on the Guthrie--Nymann Cantorval}

\begin{theorem}
Let $m=1$, that is, $s=4$. If
$
p_2=p_3=p_0+p_4=p_1+p_5=\frac{1}{4}
$
and $p_1\geq p_0$, then  $\xi$ has an absolutely continuous distribution.
\end{theorem}

\begin{proof}
We prove that under the assumptions of the theorem, $\xi$ can be represented as the sum of two independent random variables $\tau$ and $\eta$, one of which has an absolutely continuous (in fact, uniform) distribution on the interval $[0,1]$. Then, by a well-known fact~\cite{p2}, the sum $\tau+\eta$ also has an absolutely continuous distribution.

Consider the random variable
$
\tau=\sum\limits_{n=1}^{\infty}4^{-n}\tau_n,
$
where $(\tau_n)$ is a sequence of independent random variables taking the values $0,1,2,3$ with probability $\frac{1}{4}$ each (as is well known, $\tau$ has the uniform distribution on $[0,1]$), and the random variable
$
\eta=\sum\limits_{n=1}^{\infty}4^{-n}\eta_n,
$
where $(\eta_n)$ is a sequence of independent random variables taking the values $0,1,2$ with probabilities $u$, $v$, and $1-u-v$, respectively (it has a singular distribution on $[0,1]$). Then their sum
$
\tau+\eta=\theta=\sum\limits_{n=1}^{\infty}4^{-n}\theta_n
$
is a random variable whose digits $\theta_n$ in the $\Delta$-representation take the values $0,1,2,3,4,5$ with probabilities
$
p_0=\frac{1}{4}u,
p_1=\frac{1}{4}u+\frac{1}{4}v,
p_2=\frac{1}{4}=p_3,
p_4=\frac{1}{4}-p_0,
p_5=\frac{1}{4}-p_1.
$
Since under the assumptions of the theorem this system is consistent (it has a nonnegative solution with respect to $u$ and $v$), we obtain that $\xi=\tau+\eta$ has an absolutely continuous distribution.
\end{proof}

\begin{corollary}
If $p_0=p_2=p_3=p_5=\frac{1}{4}$, then $\xi$ has an absolutely continuous distribution whose spectrum (the minimal closed support) is the Guthrie--Nymann Cantorval.
\end{corollary}

\begin{corollary}
If one of the following conditions holds:
$
p_0=0=p_1,p_0=0=p_5, p_4=0=p_5,
$
then $\xi$ has an absolutely continuous distribution.
\end{corollary}

The above results can easily be extended to the case of an arbitrary integer $s>3$.

\begin{theorem}
The random variable $\xi$ can be represented as the sum of two independent random variables, one of which is uniformly distributed on the interval $[0,1]$, if and only if
\begin{equation}\label{um}
p_1\ge p_0
\quad \text{and} \quad
p_2=p_3=\dots=p_{s-2}=p_{s-1}=p_0+p_s=p_1+p_{s+1}=\frac{1}{s}.
\end{equation}
\end{theorem}

\begin{proof}
Let us consider a random variable
$
\tau=\sum\limits_{n=1}^{\infty}s^{-n}\tau_n,
$
where $(\tau_n)$ is a sequence of independent random variables taking the values $0,1,\dots,s-1$ with probability $\frac{1}{s}$ each (as is well known, $\tau$ has the uniform distribution on $[0,1]$), and the random variable
$
\eta=\sum\limits_{n=1}^{\infty}s^{-n}\eta_n,
$
where $(\eta_n)$ is a sequence of independent random variables taking the values $0,1,2$ with probabilities $u$, $v$, and $1-u-v$, respectively (it has a singular Cantor-type distribution on $[0,1]$). Then their sum
$
\tau+\eta=\theta=\sum\limits_{n=1}^{\infty}s^{-n}\theta_n
$
is a random variable whose digits $\theta_n$ in the $\Delta$-representation take the values $0,1,\dots,s,s+1$ with probabilities
$
p_0=\frac{1}{s}u,
p_1=\frac{1}{s}u+\frac{1}{s}v,
p_2=\frac{1}{s}=p_3=\dots=p_{s-1},
p_s=\frac{1}{s}-p_0,
p_{s+1}=\frac{1}{s}-p_1.
$
Since under the assumptions of the theorem this system is consistent (it has a solution with respect to $u$ and $v$), we obtain that $\xi=\tau+\eta$ has an absolutely continuous distribution.
\end{proof}

\section{Conditions for singularity of the distribution: the method of characteristic functions}

To obtain sufficient conditions for the singularity of the distribution of $\xi$, we use the method of characteristic functions.
Recall that the \emph{characteristic function} of the distribution of a random variable $\zeta$ with distribution function $F_{\zeta}(x)$ is the complex-valued function $f_{\zeta}(t)$ defined as the expectation of the random variable $e^{it\zeta}$, i.e.,
$$
f_{\zeta}(t)=Me^{it\zeta}=\int\limits_{-\infty}^{+\infty}e^{itx}\,dF_{\zeta}(x).
$$

It is known~\cite{p2} that for an absolutely continuous distribution of a random variable $\zeta$ with characteristic function $f_{\zeta}(t)$, the number
$
L_{\zeta}\equiv\lim\limits_{|t|\to\infty}\sup |f_{\zeta}(t)|
$
is equal to $0$; for a discrete distribution, $L_{\zeta}=1$; and for a singular distribution $L_{\zeta}\in [0,1]$. For a singular distribution, the value of $L_{\zeta}$ characterizes its closeness to either an absolutely continuous or a discrete distribution. For a continuous random variable $\zeta$ with a pure Lebesgue type distribution, the condition $L_{\zeta}\neq 0$ implies singularity. This is the essence of the method of characteristic functions for proving singularity of a distribution.

Since $(\xi_n)$ is a sequence of independent random variables and
$
\xi=\frac{\xi_1}{s}+\frac{\hat{\xi}}{s},
$
where
$
\hat{\xi}=\frac{\xi_2}{s}+\frac{\xi_3}{s^2}+\dots,
$
we conclude that $\xi_1$ and $\hat{\xi}$ are independent, and $\hat{\xi}$ has the same distribution as $\xi$. Thus, for the characteristic functions of $\xi$ and $\hat{\xi}$ we have
$
f_{\xi}(t)=f_{\hat{\xi}}(t).
$
By definition, we get
$$
f_{\xi}(t)=Me^{it\xi}
=Me^{it\sum\limits_{k=1}^{\infty}s^{-k}\xi_k}
=Me^{it\frac{\xi_1}{s}}\cdot Me^{i\frac{t}{s}\hat{\xi}}
=\varphi_1(t)f_{\xi}\!\left(\frac{t}{s}\right),
$$
where
$$
\varphi_1(t)=Me^{it\frac{\xi_1}{s}}
=p_0+p_1e^{i\frac{t}{s}}+p_2e^{i\frac{t}{s}\cdot 2}+\dots+p_{s+1}e^{i\frac{t}{s}(s+1)}.
$$
Therefore, if $\varphi_1(t_0)=1$, then
$
f_{\xi}(t_0)=f_{\xi}\!\left(\frac{t_0}{s}\right).
$

Since $\varphi_1(2\pi s^n)=1$, it follows that
$$
f_{\xi}(2\pi)=f_{\xi}(2\pi s)= \dots = f_{\xi}(2\pi s^{n-1}) = f_{\xi}(2\pi s^n), \qquad n\in\mathbb{N}.
$$
Now we have proved the following statement.

\begin{lemma}
The characteristic function $f_{\xi}(t)$ of the random variable ~ ~ ~ ~ ~ $\xi=\Delta_{\xi_1\xi_2\dots\xi_n\dots}$ admits the representation
\[
f_{\xi}(t)=\lim\limits_{m\to\infty}\prod_{k=1}^{m}\varphi_k(t)=\prod_{k=1}^{\infty}\varphi_k(t),
\qquad
\text{where }
\varphi_k(t)=\sum\limits_{m=0}^{s+1}p_me^{i\frac{t}{s^k}m},
\]
and for every $n\in \mathbb{N}$ satisfies the relation
$
f_{\xi}(2\pi)=f_{\xi}(2\pi s^n).
$
\end{lemma}

\begin{corollary}
The following inequality holds
$$
L_{\xi}\equiv \lim\limits_{|t|\to \infty}\sup|f_{\xi}(t)|\geq |f_{\xi}(2\pi)|.
$$
\end{corollary}

\begin{theorem}\label{th4}
If $m=1$, that is, $s=4$, and at least one of the conditions
$
p_2\neq p_0+p_4
$
or
$
p_3\neq p_1+p_5
$
holds, then $\xi$ has a singular distribution on $\left[0,\frac{5}{3}\right]$.
\end{theorem}

\begin{proof}
We apply the method of characteristic functions.
For $s=4$ we have
\[
\varphi_k(t)=\left(p_0+\sum\limits_{m=1}^{5}p_m\cos\frac{t}{4^k}m\right)+i\sum\limits_{m=1}^{5}p_m\sin\frac{t}{4^k}m,
\]
and $\varphi_k(2\pi)\neq 0$ for $k>1$, while
$
\varphi_1(2\pi)=(p_0-p_2+p_4)+i(p_1-p_3+p_5).
$
Thus, $\varphi_1(2\pi)=0$ is equivalent to 
$
p_2=p_0+p_4,
p_3=p_1+p_5.
$
Next we will show that if these conditions fail, then $|f_{\xi}(2\pi)|\neq 0$, and hence $L_{\xi}>0$.

Consider the product
\begin{equation}\label{prod1}
\prod\limits_{k=2}^{\infty}\varphi_k(2\pi).
\end{equation}
Its convergence is equivalent to the convergence of the series
\begin{equation}\label{sum1}
\sum\limits_{k=2}^{\infty}|\varphi_k(2\pi)-1|.
\end{equation}
Let us estimate a general term of this series
\begin{align*}
|\varphi_k(2\pi)-1|
&=\left[\left(p_0-1+\sum\limits_{j=1}^{5}p_j\cos\frac{2\pi}{4^k}j\right)^2+
\left(\sum\limits_{j=1}^{5}p_j\sin\frac{2\pi}{4^k}j\right)^2\right]^{\frac{1}{2}} \leq \\
&\leq
2\left(\sum\limits_{j=0}^{5}p_j\sin^2\frac{\pi j}{4^k}-\sum\limits_{r\neq j}^{5}p_rp_j\sin^2\frac{\pi(r-j)}{4^k}\right)^{\frac{1}{2}} \leq \\
&\leq
2\left(\sum\limits_{j=0}^{5}p_j\sin^2\frac{\pi j}{4^k}\right)^{\frac{1}{2}}
\leq
2\left(\sum\limits_{j=0}^{5}p_j\sin^2\frac{5\pi}{4^k}\right)^{\frac{1}{2}} \leq \\
&\leq
2\sin\frac{5\pi}{4^k}
\leq
2\pi\frac{5}{4^k}.
\end{align*}
Note that the last two inequalities hold for all $k>1$. Since the geometric series with general term $2\pi\frac{5}{4^k}$ converges, it follows from the comparison test that the series~\eqref{sum1} also converges, and hence the product~\eqref{prod1} converges as well. Thus, under the assumptions of the theorem we have
$
|f_{\xi}(2\pi)|>0
$
and
$
L_\xi>0.
$
Therefore, $\xi$ has a singular distribution.
\end{proof}


\begin{corollary}
If $p_1=p_4=0$ and the assumption of the theorem is fulfilled, then $\xi$ has a singular distribution on the Guthrie--Nymann Cantorval.
\end{corollary}

Indeed, the theorem implies the singularity of the distribution of $\xi$, while its spectrum is the Guthrie--Nymann Cantorval.

\begin{theorem}
If $q_0\neq \frac{1}{2}$, then the infinite Bernoulli convolution generated by the Guthrie--Nymann series has a purely singular distribution.
\end{theorem}

\begin{proof}
Consider the distribution of the random variable
$$
\zeta=\frac{3\zeta_1}{4}+\frac{2\zeta_2}{4}+\frac{3\zeta_3}{4^2}+\frac{2\zeta_4}{4^2}+\dots+\frac{3\zeta_{2k-1}}{4^k}+\frac{2\zeta_{2k}}{4^k}+\dots,
$$
where $(\zeta_n)$ is a sequence of independent random variables taking the values $0$ and $1$ with probabilities $q_0$ and $q_1$, respectively. Then the random variable
$
\eta_k=3\zeta_{2k-1}+2\zeta_{2k}
$
takes the values $0,2,3,5$ with probabilities $q_0^2$, $q_0q_1$, $q_0q_1$, and $q_1^2$, respectively. The independence of $(\zeta_n)$ implies the independence of the sequence $(\eta_k)$. Therefore, the distribution of $\zeta$ is in fact equivalent to the distribution of $\xi$ under the conditions
$
p_0=q_0^2,
p_2=q_0q_1,
p_3=q_0q_1,
p_5=q_1^2,
p_1=0=p_4.
$

Since the assumption of Theorem~\ref{th4} is fulfilled, i.e.,
$
q_0q_1=p_2\neq p_0=q_0^2
$
or
$
q_0q_1=p_3\neq p_5=q_1^2,
$
we conclude that the distribution of $\zeta$ is singular whenever $q_0\neq \frac{1}{2}$ ($q_1=1-q_0\neq \frac{1}{2}$).
\end{proof}

\section{The general case}


\begin{theorem}
Let $m>1$ and $s=2m+2$. If at least one of the following conditions 
$$
u\equiv (p_0+p_s-p_{\frac{s}{2}})+(p_1+p_{s+1}+p_{s-1})\cos\frac{2\pi}{s}+
\sum\limits_{j=2}^{m}(p_j+p_{s-j})\cos\frac{\pi j}{m+1}\neq 0,
$$
$$
v\equiv (p_1+p_{s+1}-p_{s-1})\sin\frac{\pi}{m+1}+
\sum\limits_{j=2}^{m}(p_j-p_{s-j})\sin\frac{\pi j}{m+1}\neq 0,
$$
is fulfilled, then
$
L_{\xi}\geq |f_\xi(2\pi)|>0,
$
and hence the distribution of $\xi$ is singular.
On the other hand, if $u=0=v$, then $|f_\xi(2\pi)|=0$, and the question about the type of the distribution of $\xi$ remains open.
\end{theorem}

\begin{proof}
Since the infinite product
$
\prod\limits_{k=2}^{\infty}\varphi_k(2\pi)
$
converges, we have $f_\xi(2\pi)=0$ if and only if $\varphi_1(2\pi)=0$. However,
$$
\varphi_1(2\pi)=p_0+\sum\limits_{j=1}^{s+1}{p_j}\cos\frac{2\pi j}{s}
+i \sum\limits_{j=1}^{s+1}{p_j}\sin\frac{2\pi j}{s}
=u+iv,
$$
where
$$
u=p_0+\sum\limits_{j=1}^{s+1}{p_j}\cos\frac{\pi j}{m+1}
=$$
$$=(p_0+p_s-p_{\frac{s}{2}})+(p_1+p_{s+1}+p_{s-1})\cos\frac{\pi}{m+1}+
\sum\limits_{j=2}^{m}(p_j+p_{s-j})\cos\frac{\pi j}{m+1},
$$
$$
v=\sum\limits_{j=1}^{s+1}{p_j}\sin\frac{2\pi j}{s}
=(p_1+p_{s+1}-p_{s-1})\sin\frac{\pi}{m+1}+
\sum\limits_{j=2}^{m}(p_j-p_{s-j})\sin\frac{\pi j}{m+1},
$$
because
\begin{eqnarray*}
\cos\frac{\pi (s-j)}{m+1}=\cos\left(2\pi-\frac{j\pi}{m+1}\right)=\cos\frac{j\pi}{m+1},\\
\sin\frac{\pi (s-j)}{m+1}=\sin\left(2\pi-\frac{j\pi}{m+1}\right)=-\sin\frac{j\pi}{m+1}.
\end{eqnarray*}
Thus, if $u\neq 0$ or $v\neq 0$, we have $\varphi_1(2\pi)\neq 0$ and
$
L_{\xi}\geq |f_\xi(2\pi)|>0.
$

If, however, $u=0=v$, then $\varphi_1(2\pi)=0$, and no conclusion about $L_{\xi}$ can be drawn.
\end{proof}

\begin{corollary}
A necessary condition for the absolute continuity of the distribution of $\xi$ is 
$
u=0=v.
$
\end{corollary}

\begin{corollary}
If $p_1=0=p_s$, $p_0p_2p_3\dots p_{s-1}p_{s+1}\ne 0$, and the assumptions of the theorem are satisfied, then the random variable $\xi$ has a singular distribution whose spectrum is a Cantorval.
\end{corollary}

\section{Fractal properties of the boundary of the spectrum}

Let us consider the set
$
E_{s}=\left\{ x: x=\sum_{n=1}^{\infty} {\varepsilon_n}{s^{-n}}, \ (\varepsilon_n) \in L_0\right\},
$
which is the spectrum of the random variable $\xi$ under the conditions $p_1=0=p_s$ and $p_0p_2p_3\cdots p_{s-1}p_{s+1}\ne 0$.
The topological and metric properties of the set $E_s$, for $s=2m+2$, $m \in \mathbb{N}$, were studied in~\cite{Banakiewicz}. It was proved there that $E_s$ is a Cantorval and its Lebesgue measure is equal to $1$. This measure coincides with the total length of the intervals whose union is the interior of $E_s$.

Below we extend these results to the case of arbitrary $ s \geq 4$, $s \in \mathbb{N}$, and determine the fractal dimension of the boundary of the set $E_s$.

It is easy to see that $E_s$ is the attractor of an iterated function system (IFS) consisting of $s$ similarities
$\displaystyle w_i(x)=\frac{i}{s}+\frac{x}{s}$ for $i \in A_0$ (see \cite{BSS}), and hence
$$
E_s=\bigcup\limits_{i\in A_0}w_i(E_s)=W(E_s).
$$

\begin{lemma}
\label{Property IFS}
The following properties of $W=\{ w_0,w_2,w_3,\dots,w_{s-1},w_{s+1} \}$ are easy to verify:
\begin{itemize}
\item The interval $H\equiv\left[ 0, \frac{s+1}{s-1}\right]$ is the minimal interval containing $E_s$;
\item $\min w_0(E_s)=0$, \ $\max w_0(E_s)=\frac{1}{s} \cdot \frac{s+1}{s-1}$;
\item $\min w_i(E_s)=\frac{i}{s}$, \ $\max w_i(E_s)=\frac{i}{s}+\frac{1}{s} \cdot \frac{s+1}{s-1}$ for all $2 \leq i \leq s-1$;
\item $\min w_{s+1}(E_s)=\frac{s+1}{s}$, \ $\max w_{s+1}(E_s)=\frac{s+1}{s} + \frac{1}{s} \cdot \frac{s+1}{s-1}$.
\end{itemize}
\end{lemma}

Note that the distances between cylinders satisfy
\[
\rho(\Delta_0,\Delta_2)=\rho(\Delta_{[s-1]},\Delta_{[s+1]})=\frac{s-3}{s(s-1)}>0.
\]

\begin{theorem}
The interval
$
I=\left[\frac{2}{s-1}, 1\right]
$
is contained in the set $E_s$ and is the longest interval entirely contained in $E_s$. Moreover,
$
\frac{2}{s-1}=\Delta_{(2)}=\Delta_{(2)}^s,
1=\Delta_{(s-1)}=\Delta_{(s-1)}^s.
$
\end{theorem}

\begin{proof}
Recall that every number $x \in [0, 1]$ has a classical base-$s$ expansion
\begin{equation}
\label{s-expansion}
x=\Delta_{a_1 a_2 \dots a_n \dots}^{s}=\frac{a_1}{s}+\frac{a_2}{s^2}+\dots+\frac{a_n}{s^n}+\dots, \qquad a_n \in \{ 0, 1, 2, \dots, s-1\}.
\end{equation}

We first show that the interval $\left[ \frac{3}{s},1\right]$ is entirely contained in $E_s$. Let we have an arbitrary
$x \in \left[ \frac{3}{s},1\right]$. Then it has a base-$s$ expansion
$
x=\Delta_{3 a_2 \ldots a_n \dots}^{s}.
$
Next we show that there exists a sequence $(b_n) \in L_0$ such that
$$
x=\Delta_{3 a_2 \dots a_n \dots}^{s}
=\frac{b_1}{s}+\frac{b_2}{s^2}+\dots+\frac{b_n}{s^n}+\dots
=\Delta_{b_1 b_2 \dots b_n \dots}.
$$

If the sequence $(a_n)$ contains no digit $1$, then we simply set $b_1=3$ and $b_n=a_n$ for $n>1$. In the case where the sequence $(a_n)$ contains some digits equal to $1$, we describe alternative replacements of digit blocks that produce another representation of the same number using only digits from the alphabet $A_0$. We begin with the case of replacing the period $(1)$ in the representation of a number. It is clear that
$
x=\Delta_{3 a_2 \dots a_m (1)}^{s}=\Delta_{3a_2\dots a_m(0[s+1])}.
$

From now on, we assume that the representation of $x$ does not end with the period $(1)$.

1. If $a_2=1$, then the following substitutions eliminate digit 1 in the representation:
\[
\Delta^s_{311\ldots1a_{2k+2}a_{2k+3}\ldots}=
\Delta_{30[s+1]0[s+1]\ldots0[s+1]a_{2k+2}a_{2k+3}\ldots},
\]
\[
\Delta^s_{311\ldots1a_{2k+1}a_{2k+2}\ldots}
=
\Delta_{2[s+1]0[s+1]0[s+1]\ldots0[s+1]a_{2k+1}a_{2k+2}\ldots}.
\]

2. If $a_2\neq 1$, then we have
\[
\Delta^s_{301a_4a_5\ldots}=\Delta_{2[s-1][s+1]a_4a_5\ldots},
\qquad
\Delta^s_{3a_21a_4a_5\ldots}=\Delta_{3[a_2-1][s+1]a_4a_5\ldots},
\quad \text{if } a_2\neq 0.
\]

The following substitutions eliminate the first occurrence of the digit $1$ in the representation:

$
\Delta^s_{3a_2\ldots a_m1a_{m+2}a_{m+3}\ldots}
=
\Delta_{3a_2\ldots a_{m-1}[a_m-1][s+1]a_{m+2}a_{m+3}\ldots};
$

$
\Delta^s_{3a_2\ldots a_m01a_{m+3}a_{m+4}\ldots}
=
\Delta_{3a_2\ldots a_{m-1}[a_m-1][s-1][s+1]a_{m+3}a_{m+4}\ldots},
~ a_m\neq 0;
$

$
\Delta^s_{3a_2\ldots a_k0\ldots01a_{k+1+p}a_{k+2+p}\ldots}
=
\Delta_{3a_2\ldots a_{k-1}[a_k-1][s-1]\ldots[s-1][s+1]a_{k+1+p}\ldots},
a_k\neq 0;
$

$
\Delta^s_{30\ldots 01a_{k+3}a_{k+4}\ldots}
=
\Delta_{2[s-1][s-1]\ldots[s-1][s+1]a_{k+3}a_{k+4}\ldots}.
$

After eliminating the first block with $1$, the second block becomes the first one, and we proceed by the same algorithm. Continuing inductively, we conclude that for every
$x \in \left[ \frac{3}{s},1\right]$ there exists a representation that does not use the digits $1$ and $s$. Therefore,
$
\left[ \frac{3}{s}, 1\right] \subset E_s.
$

For $s \geq 4$ the inequality
$
\frac{3}{s} < \frac{s+1}{2(s-1)} < 1
$
holds, and hence
$
\left[ \frac{s+1}{2(s-1)}, 1\right] \subset \left[ \frac{3}{s}, 1\right] \subset E_s.
$
Since $E_s$ is symmetric with respect to its midpoint $\frac{s+1}{2(s-1)}$, it follows that
\[
\left[ \frac{s+1}{2(s-1)} - \left( 1 - \frac{s+1}{2(s-1)}\right), 1\right]
=
\left[ \frac{2}{s-1}, 1\right]
\subset E_s.
\]

We now prove that this interval is maximal among all intervals contained in $E_s$. To do so, we show that every left neighborhood
$(\Delta_{(2)}^s -\varepsilon; \Delta_{(2)}^s)$
of the point $\Delta_{(2)}^s$ contains an interval disjoint from $E_s$. Since
$
\Delta_0=\left[0,\frac{1}{s(s-1)}\right],
\Delta_2=\left[\frac{2}{s},\frac{2}{s}+\frac{1}{s(s-1)}\right],
$
and $\Delta_0\cap\Delta_2=\varnothing$, we obtain
$
\left(\frac{1}{s(s-1)},\frac{1}{s}\right)\cap E_s=\varnothing.
$

Since
$
\Delta_{\underbrace{2...20}_k}=\left[a_k,a_k+\frac{1}{s^k(s-1)}\right],
$
where
$
a_k=\sum\limits_{i=1}^{k-1}\frac{2}{s^i},
$
and
$
\Delta_{\underbrace{2...20}_k}\cap\Delta_{\underbrace{2...20}_{k+1}}=\varnothing,
$
we get
$
\left(a_k+\frac{1}{s^k(s-1)},a_k+\frac{1}{s^k}\right)\cap E_s=\varnothing.
$
Since $a_k\to \frac{2}{s-1}$ as $k\to\infty$, it is easy to choose $k(\varepsilon)$ such that the gap
$
\left(a_k+\frac{1}{s^k(s-1)},a_k+\frac{1}{s^k}\right)
$
of $E_s$ lies entirely inside the prescribed left neighborhood of the point $\frac{2}{s-1}$. By symmetry, an analogous situation occurs to the right of $1$, namely there exist gaps symmetric to the above ones with respect to the point $\frac{s+1}{2(s-1)}$.
\end{proof}

\begin{corollary}
The following equality holds:
\[
I=\Delta_3\cup\Delta_4\cup\cdots\cup\Delta_{[s-2]}\cup\left[\bigcup\limits_{k=1}^{\infty}
\Delta_{\underbrace{2\ldots2}_k3}\right]\cup\left[\bigcup\limits_{k=2}^{\infty}\Delta_{\underbrace{[s-1]\ldots[s-1]}_k[s-2]}\right].
\]
\end{corollary}

\begin{theorem}
\label{T:N-Self-Similarity}
The set $E_s \setminus int I$ is a countable union of pairwise disjoint affine copies of $E_s$, namely,
\begin{equation}\label{eq:fpmv}
E_s=\left(\frac{2}{s-1},1\right)\cup\bigcup\limits_{i=0}^{\infty}[\varphi_i(E_s)\cup h(\varphi_i(E_s))],
\end{equation}
where
$
\varphi_i(x=\Delta_{\alpha_1\alpha_2\cdots\alpha_n\cdots})
=
\Delta_{\underbrace{2\ldots2}_i0\alpha_1\alpha_2\cdots\alpha_n\cdots}
$
is the right-shift operator on digits, and
$
h(x=\Delta_{\alpha_1\alpha_2\cdots\alpha_n\cdots})
=
\Delta_{[s+1-\alpha_1]\cdots[s+1-\alpha_n]\cdots}
=
\frac{s+1}{s-1}-x
$
is the digit inversion of the $\Delta$-representation.
\end{theorem}

\begin{proof}
Note that $\varphi_i(x)$ is a similarity transformation with ratio
$
k_i=\frac{1}{s^{i+1}},
$
and hence $E_s$ and $\varphi_i(E_s)$ are similar. The mapping $h(x)$ is the central symmetry with respect to the point $\frac{s+1}{2(s-1)}$ of the set $E_s$. Therefore, $\varphi_i(E_s)$ and
$h(\varphi_i(E_s))$ are isometric figures.

Taking into account the consequences of the previous lemma and theorem, we have
\[
\varphi_i(E_s)=\Delta_{\underbrace{2\ldots20}_i}\cap E_s,
\qquad
h(\varphi_i(E_s))=\Delta_{\underbrace{[s-1]\ldots[s-1][s+1]}_i}\cap E_s.
\]
The set $E_s \subset H$ is the attractor of the IFS $W$, and by
Lemma~\ref{Property IFS},
$$
w_0(H) \cap w_i(H) = \varnothing,\qquad
w_i(H) \cap w_{s+1}(H) = \varnothing
\quad \text{for } 2 \leq i \leq s;
$$
$$
w_0(H) \cap I = \varnothing,
\qquad
w_{s}(H) \cap I = \varnothing.
$$

Using again the consequences of the previous lemma and theorem, together with the identities
\[
\varphi_i(E_s)=\Delta_{\underbrace{2\ldots20}_i}\cap E_s,
\qquad
h(\varphi_i(E_s))=\Delta_{\underbrace{[s-1]\ldots[s-1][s+1]}_i}\cap E_s,
\]
and
\[
\rho(\Delta_{\underbrace{2\ldots20}_i},\Delta_{\underbrace{2\ldots20}_{i+1}})
=
\frac{s-2}{s^i(s-1)}>0,
\]
we obtain~\eqref{eq:fpmv}.
\end{proof}

\begin{theorem}
The interior $int E_s$ of $E_s$ has Lebesgue measure equal to $1$.
\end{theorem}

\begin{proof}
It is well known that if the sets $D_1$ and $D_2$ are similar with similarity ratio $k$, then their one-dimensional Lebesgue measures satisfy
$
\lambda(D_2)=k\lambda(D_1).
$
Taking into account Theorem~\ref{T:N-Self-Similarity}, and equality~\eqref{eq:fpmv} in particular, we obtain
$$
\lambda(\operatorname{int} E_s)
=
\frac{s-3}{s-1}
+\frac{2}{s}\lambda(int E_s)
+\frac{2}{s^2}\lambda(int E_s)
+\frac{2}{s^3}\lambda(int E_s)
+\cdots
=
$$
$$
=
\frac{s-3}{s-1}
+\frac{2}{s-1}\lambda(int E_s).
$$
Hence
$
\left(1-\frac{2}{s-1}\right)\lambda(\operatorname{int} E_s)=\frac{s-3}{s-1},
$
that implies
$
\lambda(\operatorname{int} E_s)=1.
$
\end{proof}

\begin{theorem}
The boundary $fr E_s$ of the Cantorval $E_{s}$ is an $N$-self-similar set with Hausdorff--Besicovitch dimension $\log_{s}{3}$.
\end{theorem}

\begin{proof}
The boundary
$
fr E_s=E_s\setminus int E_s
$
of the set $E_s$ is an $N$-self-similar set (it is a union of pairwise disjoint affine copies of itself) since
\[
fr E_s=\bigcup\limits_{i=1}^{\infty}[\Delta_{\underbrace{2...20}_i}\cap fr E_s]\cup\left[\bigcup\limits_{i=1}^{\infty}\Delta_{\underbrace{[s-1]...[s-1][s+1]}_{i}}\cap fr E_s\right],
\]
and the distances satisfy
$$
\rho(\Delta_{\underbrace{2\ldots20}_i},\Delta_{\underbrace{2\ldots20}_{i+1}})
=
\rho(\Delta_{\underbrace{[s-1]...[s-1][s+1]}_i},\Delta_{\underbrace{[s-1]...[s-1][s+1]}_{i+1}})
=$$
$$
=
\frac{s-2}{s^i(s-1)}>0,
$$
while $fr E_s$ is similar to
$
\Delta_{\underbrace{2...20}_{i}}\cap fr E_s
=
\Delta_{\underbrace{[s-1]...[s-1][s+1]}_{i}}\cap fr E_s
$
with similarity coefficient
$
k_i=\frac{1}{s^{i}}.
$

Therefore, its Hausdorff--Besicovitch dimension is equal to its $N$-similarity dimension, which is the solution of the equation
\[
2\sum\limits_{i=1}^{\infty}k_i^x=1
\Longleftrightarrow
2\sum\limits_{i=1}^{\infty}s^{-ix}=1,
\]
namely,
$
\dim_H fr E_s=\log_{s}{3}.
$
\end{proof}

\begin{corollary}
The boundary of the Guthrie--Nymann Cantorval (the case $s=4$) has Hausdorff--Besicovitch dimension equal to $\log_{4}{3}$.
\end{corollary}

\end{article}

\ifx\ArticleWithinIssue\undefined\newpage\authortitleforheadfootcontent{}{}\thispagestyle{empty}\setcounter{tocdepth}{0}\renewcommand{\contentsname}{\centerline{CONTENTS}}\tableofcontents\end{document}\fi 